\documentclass{article}

\title{Generating Functions of Nestohedra and Applications}
\author{Andrew G. Fenn\footnote{The author would like to thank the EPSRC for his
funding.}}

\usepackage{ams}
\usepackage{amsfonts}
\usepackage{amsmath}
\usepackage{amsthm}
\usepackage{amssymb}
\usepackage{amscd}
\usepackage{stmaryrd}
\usepackage{graphics}
\usepackage{color}
\usepackage{graphpap}
\usepackage[paper = a4paper,left=2.5cm,top=3cm,right=2.5cm,bottom =
3.25cm,includehead]{geometry}

\usepackage{fancyhdr}

\input xy
\xyoption{all}

\theoremstyle{definition}
\newtheorem{defn}{Definition}[section]
\newtheorem{thm}[defn]{Theorem}

\newtheorem{lem}[defn]{Lemma}

\newtheorem{cor}[defn]{Corollary}

\newtheorem{exm}{Example}
\newtheorem{con}[defn]{Conjecture}

\newcommand{\nat}{\ensuremath{\mathbb N}}
\newcommand{\rat}{\ensuremath{\mathbb Q}}

\newcommand{\perm}[1]{\ensuremath{Pe^{#1}}}

\newcommand{\permu}{\ensuremath{Pe}}
\newcommand{\astr}[1]{\ensuremath{St^{#1}}}
\newcommand{\astro}{\ensuremath{St}}

\newcommand{\gengra}[4]{\ensuremath{\Gamma_{#1,#2,#3,#4}}}
\newcommand{\genpol}[4]{\ensuremath{P_{#1,#2,#3,#4}}}

\newcommand{\ggkk}[2]{\ensuremath{\Gamma_{\nabla,#1,\nabla,#2}}}
\newcommand{\gpkk}[2]{\ensuremath{P_{\nabla,#1,\nabla,#2}}}

\newcommand{\ggke}[2]{\ensuremath{\Gamma_{\nabla,#1,\because,#2}}}
\newcommand{\gpke}[2]{\ensuremath{P_{\nabla,#1,\because,#2}}}

\newcommand{\gpkea}{\ensuremath{P_{{\nabla,\because}}}}
\newcommand{\gpkeaf}{\ensuremath{P_{f,{\nabla,\because}}}}
\newcommand{\gpkeah}{\ensuremath{P_{h,{\nabla,\because}}}}

\newcommand{\ggkeo}[1]{\ensuremath{\Gamma_{\nabla,#1,\because}}}
\newcommand{\gpkeo}[1]{\ensuremath{P_{\nabla,#1,\because}}}
\newcommand{\gpkeof}[1]{\ensuremath{P_{f,{\nabla,#1,\because}}}}
\newcommand{\gpkeoh}[1]{\ensuremath{P_{h,{\nabla,#1,\because}}}}

\newcommand{\ggek}[2]{\ensuremath{\Gamma_{\because,#1,\nabla,#2}}}
\newcommand{\gpek}[2]{\ensuremath{P_{\because,#1,\nabla,#2}}}

\newcommand{\ggee}[2]{\ensuremath{\Gamma_{\because,#1,\because,#2}}}
\newcommand{\gpee}[2]{\ensuremath{P_{\because,#1,\because,#2}}}

\newcommand{\ggeea}{\ensuremath{\Gamma_{{\because,\because}}}}
\newcommand{\gpeea}{\ensuremath{P_{{\because,\because}}}}
\newcommand{\gpeeaf}{\ensuremath{P_{f,{\because,\because}}}}
\newcommand{\gpeeah}{\ensuremath{P_{h,{\because,\because}}}}
\newcommand{\gpeeag}{\ensuremath{P_{\gamma,{\because,\because}}}}

\newcommand{\ptpe}{\ensuremath{\mathcal{P}}}

\numberwithin{equation}{section}

\begin{document}

\maketitle

\begin{abstract}
We examine the family of nestohedra resulting from the complete
bipartite graph through the medium of a generating function and
demonstrate some of their combinatorial invariants.
\end{abstract}

\section{Introduction}
In \cite{VB} a ring of simple polytopes \ptpe\ was introduced, with
a formal addition akin to disjoint union, and direct product as
multiplication.  This ring inherits all the machinery that is
normally associated with polytopes.  In particular, the $f$-, $h$-
and $\gamma$-vectors are defined on the generators of this ring.
These vectors are extended, again in \cite{VB}, to polynomials with
the entries of the vectors as coefficients, known as the $f$-, $h$-
and $\gamma$-polynomials, which we will favour here. Also introduced
in \cite{VB} was a derivation, $d$, which produces the disjoint
union of the facets of a polytope. In the same paper the concept of
a generating function for a family of polytopes was developed
together with calculations of the $f$- and $h$-polynomials of some
well-known families. In this paper we are going to describe the
combinatorial invariants of some additional families of simple
polytopes. The particular polytopes involved in these examples are
constructed as in \cite{PRW}, so an overview of this construction
will also be included in this paper.  The main results of this paper
will be describing the invariants of the family of polytopes
resulting from the complete bipartite graphs and also the method
used to obtain them.

\section{The Basics}
A \emph{family of polytopes} is a collection of polytopes that share
some defining property. A family has at least one representative in
each dimension. It is an indexed by a set $J$. For example, we have
the family, $I=\{I^n\}_{n\in\nat}$, consisting of all cubes, and the
family, $\Delta=\{\Delta^n\}_{n\in\nat}$, consisting of all
simplices.

\begin{defn}\label{genfundef}
For a family of polytopes $\Psi$, with indexing set $J$, we define
the \emph{generating function} as the formal power series
$$\Psi(x) := \sum_{j \in J}s(j)\,P^n_j\, x^{n+q}$$
in $\ptpe\otimes \rat[[x]]$. In this series, the parameters $s(j)
\in \rat$ and $q \in \nat$ are chosen appropriately for the family
$\Psi$ in question, to simplify later equations.
\end{defn}


For certain families of polytopes we can choose a function $s(i)$
which depends directly on the index of the polytope. We can then
choose a $q$ which removes later correcting factors from the
differential equations, allowing the generating function to be
studied independently from the individual polytopes. All the
families studied in this paper will have this desirable property.
The choice of $s(P^{n})$ and $q$ depends on some subgroup of the
group of symmetries of the polytopes and will become more apparent
with some examples.

\begin{exm}
The first example to consider is $I$, the family of cubes. The
symmetry group of $I^{n}$, without reflections, has order $n!$, so
we choose $s(I^{n}) = 1 /n!$. We then choose $q$ to be $0$, giving
$\sum I^n x^{n} / n!$ as the generating function. We make this
choice so that when the generating function is differentiated with
respect to $x$ we get a simple expression.
\end{exm}

\begin{exm}
Another pertinent example is that of $\Delta$, the family of
simplices. In this case, the symmetry group of $\Delta^{n}$ has
order $(n+1)!$, and hence we choose $s(\Delta^{n})= 1/(n+1)!$. In
this case, we set $q=1$, giving $\sum \delta^n x^{n+1} / (n+1)!$ as
the generating function.
\end{exm}

We can extend the machinery that exists on \ptpe\ by linearity to
$\ptpe\otimes\rat[[x]]$; in particular we will make use of the
derivative $d$, the $f$-polynomial, $h$-polynomial and
$\gamma$-polynomial, all from \cite{VB}.  Their extensions are, for
a family $\Psi=\{P^n_j\}_{j\in J}$, respectively,
\begin{eqnarray*}
d\Psi(x)&=&\sum_{j \in J}s(j)\,d(P^n_j)\, x^{n+q} \in \ptpe\otimes\rat[[x]],\\
\Psi_f(\alpha,t,x)&=&\sum_{j \in J}s(j)\, f(P^n_j)(\alpha,t)\, x^{n+q} \in\rat[\alpha,t][[x]],\\
\Psi_h(\alpha,t,x)&=&\sum_{j \in J}s(j)\, h(P^n_j)(\alpha,t)\, x^{n+q} \in\rat[\alpha,t][[x]],\\
\Psi_\gamma(\tau,z)&=&\sum_{j \in J}s(j)\, \gamma(P^n_j)(\tau)\,
z^{n+q} \in\rat[\tau][[x]].
\end{eqnarray*}

In particular, we can restate the identity from Theorem 1 in
\cite{VB} as
\begin{equation}\label{did}(d\Psi)_f(\alpha,t,x)=\frac{\partial}{\partial t}\Psi_f(\alpha,t,x).\end{equation}
Similarly, the coordinate changes which relate the $h$-polynomial to
the $f$- and $\gamma$-polynomials are unaffected by the extension to
generating functions.  Setting $a=\alpha + t$, $b=\alpha t$,
$\tau=\frac{b}{a^2}$ and $z=ax$, the coordinate changes are
\begin{eqnarray*}
\Psi_f(\alpha,t,x) &=& \Psi_h(\alpha-t,t,x)\\
a^q \Psi_\gamma(\tau,z) &=& \Psi_h(\alpha,t,x).
\end{eqnarray*}

Now we overview the construction of Nestohedra from \cite{PRW}.
\begin{defn}
A \emph{building set}, $B$, is a set of subsets of $[n+1]:=
\{1,\ldots,n+1\}$, the set consisting of the first $n+1$ integers,
such that
\begin{enumerate}
\item{if $S_1,S_2\in B$ such that $S_1 \cap S_2 \neq \Phi$ then $S_1 \cup S_2 \in B$,}
\item{the set $\{i\}\in B$ for all $i \in [n+1]$.}
\end{enumerate}
A building set is called \emph{connected} if $[n+1] \in B$.
\end{defn}

For a graph, $\Gamma$, on $n+1$ nodes any numbering of then nodes
produces a building set $B(\Gamma)$. This building set consists of
all non-empty subsets, $I \subset [n+1]$, such that the induced
graph $\Gamma|_I$ is connected.  A building set constructed from a
graph will be called a \emph{graphical building set}. A connected
graph will produce a connected building set.

\begin{defn}
For a building set $B$, the \emph{nestohedron}, $P_B$, is the
Minkowski sum
$$P_B = \sum_{S\in B} \Delta_S$$
where $\Delta_S := \mathrm{ConvexHull}\{e_i|i\in S\}$ and $e_i$ is
the tip of the standard unit basis vector.
\end{defn}

Note that \cite{PRW} proves that nestohedra are always simple and
that all graphical nestohedra are flag. We can also generalise a
result of \cite{PRW} to give us an expression for the differential
of a nestohedron.

\begin{thm}[\cite{AF}]\label{mylem}
For a nestohedra, $P_B$, on a connected building set $B$, we have;
$$d(P_B) = \sum_{S\in B\slash[n+1]} P_{B|S} \times P_{B-S}$$
where $B|S$ is the building set consisting of those sets in $B$
which are subsets of $S$ and $B-S$ is the building set consisting of
sets in $B$ with the elements of $S$ removed.
\end{thm}
\begin{proof}
We are looking at the facets of a Minkowski sum of faces of the
standard simplex, $\Delta^n$, which includes the standard simplex as
one of the summands. All the summands will be simplices of dimension
$m\leq n$ and all their faces will also be lower dimensional
simplices. The standard definition of the $m$-dimensional simplex is
$$\Delta^m = \left\{x=(x_1,\ldots,x_m):0\leq x_i\leq 1, \sum_{i=1}^{m} x_i=1\right\},$$
and the faces of $\Delta^m$ are subsets where some set of
coordinates are minimised.

A facet of the sum will have a contribution from each summand. This
contribution will be a face of the summand, frequently the entire
summand. Each face is defined by the set of coordinates on which it
is minimised. If these sets are not restrictions of some set $C$ the
result is not a facet of the sum, it is either in the interior of
the polytope or a face of lower dimension.

The facet, as the Minkowski sum of these faces can be split up into
the Minkowski sum of two other polytopes, $X$ and $Y$. These are
both Minkowski sums of faces, $X$ the parts of those faces in
$\mathrm{span}\{e_i\}_{i \in C}$ and $Y$ the parts of those faces in
$\mathrm{span}\{e_i\}_{i\not\in C}$. Since $A$ and $B$ are
orthogonal, we have that $X+Y=X\times Y$, the direct product.

We now examine the faces in two types, those where $\sum_{i\in
C}x_i=0$ and those where $\sum_{i\in C}x_i=1$. It is easy to show
that there are no other possibilities. A face, $\Delta_S$, of the
first type contribute 0 to $X$ and $\Delta_{S-C}$ to $Y$. A face,
$\Delta_S$, of the second type contribute $\Delta_S$ to $X$ and 0 to
$Y$.

From the above we can clearly see that $X$ is precisely $P_{B|S}$
and $Y$ is $P_{B-S}$.  The product produces a facet precisely when
both sums contain the highest possible dimension simplices. This
will only occur when $\Delta_S$ and $\Delta^n$ are present, which is
when $S$ and $[n+1]$ are distinct and contained in $B$. So for a
connected building set there is a facet for each element of $B$
apart from $[n+1]$, and it is $P_{B|S} \times P_{B-S}$.
\end{proof}

For a non-connected building set we can split it up into the product
of its connected components and use lemma \ref{mylem} result on each
component.  This result can be restated in terms of graphs for
graphical building sets as
\begin{cor}
For a connected graph $\Gamma$ on $n+1$ nodes, we have
$$d(P(\Gamma))=\sum_* P(\Gamma_G) \times P(\bar{\Gamma}_{G^c})$$
where
\begin{enumerate}
\item{$G\subsetneq\{1,\ldots,n+1\}$.}
\item{$\Gamma_G$ is the subgraph of $\Gamma$ with vertex set $G$.}
\item{$\bar{\Gamma}_{G^c}$ is the graph with vertex set $\{1,\ldots,n+1\} - G$ and arcs
between two vertices, $i$ and $j$, if they are path connected in
$\Gamma_{G\cup \{i,j\}}$.}
\item{* runs over all G such that $\Gamma_G$ is connected.}
\end{enumerate}
\end{cor}

Of particular interest to us is the fact that nestohedra naturally
form families.  We can also combine theorem \ref{mylem} and equation
(\ref{did}) to give an easy way to calculate $\Psi_f(\alpha,t,x)$
for a family $\Psi$.  This method was employed in \cite{VB} to give
results about two families which we shall use here. These families
are $\permu=\{\perm n\}_{n\in\nat}$, the family of permutohedra,
which is a family of graphical nestohedra where $\perm n$ is
generated from the complete graph on $n+1$ nodes, and
$\astro=\{\astr n\}_{n\in\nat}$, the family of stellohedra, which is
a family of graphical nestohedra where $\astr n$ is generated from
the star graph on $n+1$ nodes. These families have generating
functions which we chose to be;
\begin{eqnarray*}
\permu(x) &=& \sum_{n=0}^{\infty} \perm n \frac{x^{n+1}}{(n+1)!} \\
\astro(x) &=& \sum_{n=0}^{\infty} \astr n \frac{x^{n}}{n!}.
\end{eqnarray*}

Theorem \ref{mylem} gives formulas for $d$ of the individual
nestohedra to be;
\begin{eqnarray*}
d(\perm n) &=& \sum_{i+j=n-1} {n+1 \choose i+1} \perm i \times \perm j \\
d(\astr n) &=& n.\astr{n-1} + \sum_{i=0}^{n-1} {n \choose i} \astr i
\times \perm{n-i-1}.
\end{eqnarray*}

Combining these two gives us that;
\begin{eqnarray*}
d\permu(x) &=& \sum_{n=0}^{\infty} \left(\sum_{i+j=n-1} {n+1 \choose i+1} \perm i \times \perm j\right) \frac{x^{n+1}}{(n+1)!}\\
&=&\permu(x)^2\\
d\astro(x) &=& \sum_{n=0}^{\infty}\left(n.\astr{n-1} + \sum_{i=0}^{n-1} {n \choose i} \astr i \times \perm{n-i-1}\right)\frac{x^{n}}{n!} \\
&=& \left( x+\permu(x)\right)\astro(x).
\end{eqnarray*}

Passing to the $f$-polynomial and using equation (\ref{did}) gives
us partial differential equations;
\begin{eqnarray*}
\frac{\partial}{\partial t}\permu_f(\alpha,t,x)&=&\permu_f^2(\alpha,t,x)\\
\frac{\partial}{\partial t}\astro_f(\alpha,t,x)&=&\left(
x+\permu_f(\alpha,t,x)\right)\astro_f(\alpha,t,x).
\end{eqnarray*}

Solving these partial differential equations yields;
\begin{eqnarray*}
\permu_f(\alpha,t,x)&=&\frac{e^{\alpha x}-1}{\alpha-t\left(e^{\alpha x}-1\right)}\\
\astro_f(\alpha,t,x)&=&e^{(\alpha+t)x}\frac{\alpha}{\alpha-t\left(e^{\alpha
x}-1\right)}.
\end{eqnarray*}

\section{The Method}
One of the applications of generating functions we will demonstrate
in this paper is to show that these families of nestohedra, which as
we have noted are flag, satisfy the Gal conjecture. In this section
we will develop the methods which allow us to make these
calculations.

\begin{con}[Gal '05] \label{Gal}
For any flag simple polytope, $P$, the $\gamma$-polynomial
$\gamma(P)$ has non-negative coefficients
\end{con}

Inspired by this conjecture we shall make two definitions. Firstly
we shall fix a grading on the ring   $\rat[\alpha,t][[x]]$ where
$\deg(1)=0$, $\deg(\alpha)=\deg(t)=-2$ and $\deg(x)=2$.  Then,

\begin{defn}
A \emph{Gal series} in $\mathbb{Z}[\alpha,t][[x]]$ is an element
$\psi (\alpha,t,x)$, such that,
\begin{enumerate}
\item{$\psi (\alpha,t,x)=\psi (t,\alpha,x)=\hat\psi (a,b,x)$.}
\item{$\psi$ is homogenous under the above grading.}
\item{$\hat\psi (a,b,x)$ has all non-negative coefficients.}
\end{enumerate}
\end{defn}
A family of nestohedra, $\Psi$, satisfies the Gal conjecture
precisely when $\Psi_h(\alpha,t,x)$ is a Gal series.  Note that by
the nature of the $h$-polynomial, $\Psi_h(\alpha,t,x)$ satisfies the
fist two conditions for any family of simple polytopes. To show that
$\Psi_h(\alpha,t,x)$ is a Gal series we will employ the following
lemma.

\begin{lem}\label{keylem}
Let $\Psi$ be a family of polytopes, indexed by some set $J$. If
$\Psi_h(\alpha,t,x)$ is such that;
\begin{itemize}
\item{$\frac{\partial \Psi_h(\alpha,t,x)}{\partial x}|_{x=0}$ is a Gal series.}
\item{$\frac{\partial \Psi_h(\alpha,t,x)}{\partial x}$ is a homogeneous polynomial
$$F(a,b,\Psi_h(\alpha,t,x),S_{1,h}(\alpha,t,x),\ldots,S_{k,h}(\alpha,t,x))\in\mathbb{Z}[\alpha,t][[x]]$$
with non-negative coefficients, where the $S_{i,h}(x)$, for
$i=1,\ldots,k$, are Gal series.}
\end{itemize}
Then $\Psi_h(\alpha,t,x)$ is a Gal series and the polytopes in
$\Psi$ satisfy the Gal conjecture.
\end{lem}

\begin{proof}
Let $$\frac{\partial \Psi_h}{\partial
x}(\alpha,t,x,y)=F(a,b,\Psi_h(\alpha,t,x,y),S_{i,h}(\alpha,t,x))$$
then by applying the standard substitutions that give the
$\gamma$-polynomial we have that
\begin{equation}\label{PDGam}
\frac{\partial \Psi_\gamma}{\partial
z}(\tau,z)=F(1,\tau,\Psi_\gamma(\tau,z),S_{i,\gamma}(\tau,z)).
\end{equation}

We have
$$\Psi_\gamma(\tau,z)=\sum_{j \in J}s(j)\gamma(\Psi^n_j)z^{n+q}$$
$$S_{i,\gamma}(\tau,z)=\sum_{n=1}^{\infty}S_i^n(\alpha,t) z^{n}$$
$$\frac{\partial \Psi_\gamma}{\partial z}(\tau,z)=\sum_{j \in J}(n+q)s(j)\gamma(\Psi^n_j)z^{n+q-1}$$
where $s(\Psi^n)$ are known and positive and $S_i^n(\alpha,t)$ is a
homogeneous symmetric polynomial.

Examining equation (\ref{PDGam}) term by term in $z$ gives an
identity for $\gamma(\Psi^n)$ expressed as a polynomial with
non-negative coefficients in $\tau$, $\gamma(\Psi^m)$ for $m<n$ and
$S^m_i$, $i=1,\ldots,k$ for $m\leq n$.  Thus since $S^m_{i,\gamma}$
has non-negative coefficients of $\tau$, if $\gamma(\Psi^m)$ has all
non-negative coefficients of $\tau$, for all $m<n$, then
$\gamma(\Psi^n)$ has all non-negative co-efficient of $\tau$.

Since $\frac{\partial \Psi_h}{\partial x}|_{x=0}$ is Gal, so is
$\frac{\partial \Psi_\gamma}{\partial x}|_{x=0} = \gamma(\Psi^1)$.
Thus, by induction, $\gamma(\Psi^n)$ has non-negative coefficients
for all n.  Consequently $\Psi_\gamma(\tau,z)=\hat \Psi_h(a,b,x)$
has all non-negative coefficients and $\Psi_h(\alpha,t,x,y)$ meets
the third property for being a Gal series.  Since it is related to
$\Psi(x,y)$ by the $h$-polynomial, it automatically meets the other
two conditions and so is a Gal series.
\end{proof}

\section{The Families}
Also in this paper we will calculate the $f$- and $h$-polynomials
and demonstrate that the Gal conjecture holds for another family of
graphical nestohedra. The family we will consider will be those
nestohedra generated by the complete bipartite graphs $K_{m,n}$.
However to make these calculations we will have to study some other
families first. In this section we will define these families.

The graphs producing these families will take the form of the join
of two graphs.

\begin{defn}
For two graphs $X^m$ and $Y^n$ we obtain the \emph{join},
$X+Y=\gengra XmYn$, from $X \cup Y$ by adding in all edges between
$X$ and $Y$.
\end{defn}

We shall denote  the resultant nestohedra as \genpol XmYn.  The
complete bipartite graph, $K_{m,n}$, is the join of the graph on $m$
nodes with no edges and the graph on $n$ nodes with no edges, so we
will denote $K_{m,n}$ by $\ggee mn$, the individual nestohedra by
$\gpee mn$ and $\gpeea$ will represent the entire family.  The other
families we shall look at will be generated by $\ggke mn$, $\ggek
mn$ and $\ggkk mn$ where $\nabla$ represents the complete graph.
However $\ggkk mn$ is the complete graph on $m+n$ nodes, so $\gpkk
mn = \perm {m+n-1}$ and $\ggke mn =\ggek nm$. As such we only need
to do the calculations for the families $\gpkea$ and $\gpeea$.

To preform these calculations we need to have generating functions
for these families.  We define these generating functions to be;
\begin{eqnarray*}
\gpkea(x,y) &=& \sum_{k=0}^{\infty}\sum_{l=0}^{\infty} \gpke kl \frac{x^{k}}{k!}\frac{y^{l}}{l!} \\
\gpeea(x,y) &=& \sum_{k=0}^{\infty}\sum_{l=0}^{\infty} \gpee kl
\frac{x^{k}}{k!}\frac{y^{l}}{l!}.
\end{eqnarray*}
These generating functions have two variables $x$ and $y$, rather
than just $x$. However, the addition of an extra variable is
consistent with the motivation behind definition \ref{genfundef} and
all the machinery used above extends by linearity to the two
variable case.

\section{The First Family, \gpkeo 1=\astro *}
We will now apply our method to \ggkeo 1(x,y). We notice that this
family of graphs is featured in \cite{VB} as the stellohedron
\astro(x), and we have that
\begin{eqnarray*}
\astro(x)&=&\sum_{n\geq0}\astr n \frac{x^n}{n!}\\
d\astro(x)&=&(x+\permu(x))\astro(x)\\
\astro_f(\alpha,t,x)&=&e^{(\alpha+t)x}\frac{\alpha}{\alpha-t(e^{\alpha
x}-1)}
\end{eqnarray*}
which were obtained using the same method employed here.

However, this generating function does not match our general
generating function for the join of a complete graph and an empty
graph. Fortunately, the difference in the generating functions is a
constant factor of $y$.  So we can extend the above results to
$$\gpkeof 1(\alpha,t,x,y)=e^{(\alpha+t)x}\frac{\alpha
y}{\alpha-t(e^{\alpha x}-1)}.$$

We know from \cite{PRW} that each individual stellohedron satisfies
the Gal conjecture since it is generated by a chordal graph; that is
one with no induced cycles of length 4 or more. However, we would
like to show that our method works independently of this result, so
we shall apply it to the series generated by the stellohedra.

\begin{thm}\label{gpkeoisGal}
The series $\gpkeoh 1(\alpha,t,x,y)$ is a Gal series.
\end{thm}
\begin{proof}
We will use \ref{keylem} for this. We begin by calculating the
partial derivative of the $h$-polynomial with respect to $x$
\begin{eqnarray*}
\gpkeoh 1(\alpha,t,x,y)&=& \gpkeof 1(\alpha-t,t,x) = e^{(\alpha+t)x}\frac{(\alpha-t)y}{\alpha e^{tx}-te^{\alpha x}} \\
\frac{\partial}{\partial x}\gpkeoh 1(\alpha,t,x,y) &=&
(\alpha+t)e^{(\alpha+t)x}\frac{(\alpha-t)y}{\alpha e^{tx}-te^{\alpha
x}}+
e^{(\alpha+t)x}\frac{(\alpha-t)y(\alpha e^{tx}-te^{\alpha x})}{(\alpha e^{tx}-te^{\alpha x})^2} \\
&=&(\alpha+t)\gpkeoh 1(\alpha,t,x,y)+\alpha
t\permu_h(\alpha,t,x)\gpkeoh 1(\alpha,t,x,y)
\end{eqnarray*}
since $\permu_h(\alpha,t,x)$ is a Gal series, $\gpkeoh
1(\alpha,t,x,y)$ fits the conditions of \ref{keylem} if and only if
$\gpkeoh 1(\alpha,0,x,y)$ has non-negative coefficients.
$$\gpkeoh 1(\alpha,0,x,y)=\sum_n \alpha^n \frac{x^n}{n!} = ye^{\alpha x},$$
does have non-negative coefficients so $\gpkeoh 1(\alpha,t,x,y)$ is
a Gal series.
\end{proof}

Thus each individual stellohedron satisfies the Gal conjecture.

\section{The Second Family, \gpkea}\label{sgpkea}

In this section we wish to extend our calculations to show that
those nestohedra generated by all possible graphs \ggke ij satisfy
the Gal conjecture.  Unlike in the previous section we must find the
$f$-polynomial of this family explicitly.  We start out by
calculating a formula for the derivative of the family.

\begin{lem}
The formula for $d$ of the series \gpkea is $d\gpkea(x,y)=\gpkea
(x,y)\left(y+\permu(x+y)\right)$.
\end{lem}
\begin{proof}

We also have from \ref{mylem} that
\begin{eqnarray*}
d\gpke st &=& t \gpke s{t-1} + \sum^{s}_{i=1}\sum^{t}_{j=0} {s
\choose i}{t \choose j} \gpke ij \times \perm{n-i-j}.
\end{eqnarray*}

Since we have two distinct indices the generating function is of two
variables.
\begin{eqnarray*}
d \gpkea &=& d \sum^{\infty}_{n=0} \sum_{k+l=n+1, k \geq 1} \gpke kl \frac{x^k}{k!}\frac{y^l}{l!} \\
&=& \sum^{\infty}_{n=0} \sum_{k+l=n+1, k \geq 1} l\gpke k{l-1}\frac{x^k}{k!}\frac{y^l}{l!} \\
&&+ \sum^{\infty}_{n=0} \sum_{k+l=n+1, k \geq
1}\sum^{k}_{i=1}\sum^{l}_{j=0} {k \choose i}{l \choose j}\gpke ij
\perm{n-i-j} \frac{x^{k}}{(k)!}\frac{y^{l}}{(l)!}.
\end{eqnarray*}

We can notice that the first sum is
\begin{eqnarray*}
\sum^{\infty}_{n=0} \sum_{k+l=n+1, k \geq 1} l\gpke
k{l-1}\frac{x^k}{k!}\frac{y^l}{l!}
&=& \sum^{\infty}_{n=0} \sum_{k+l=n+1, k \geq 1} y \gpke k{l-1}\frac{x^k}{k!}\frac{y^{l-1}}{(l-1)!}  \\
&=& y \gpkea (x,y).
\end{eqnarray*}

Now turning our attention to the second sum, setting $g=k-i$,
$h=l-j$ and utilising the identity $\sum_{i=0}^\infty \sum_{j+k=i+1,
j\geq l} a_{jk} = \sum_{j=l}^\infty\sum_{k=0}^\infty a_{jk}$
repeatedly, we have

\begin{eqnarray*}
 && \sum^{\infty}_{n=0} \sum_{k+l=n+1, k \geq 1}\sum^{k}_{i=1}\sum^{l}_{j=0} {k \choose i}{l \choose j}\gpke ij \perm{n-i-j} \frac{x^{k}}{(k)!}\frac{y^{l}}{(l)!} \\
 &=& \sum^{\infty}_{n=0} \sum_{k+l=n+1, k \geq 1}\sum^{k}_{i=1}\sum^{l}_{j=0} \frac{k!}{i!(k-i)!}\frac{l!}{j!(l-j)!}\gpke ij \perm{n-i-j} \frac{x^{k}}{(k)!}\frac{y^{l}}{(l)!} \\
 &=& \sum_{k=1}^\infty \sum_{l=0}^\infty \sum^{k}_{i=1}\sum^{l}_{j=0} \frac{1}{i!g!}\frac{1}{j!h!}\gpke ij \perm{g+h-1} x^{k}y^{l} \\
 &=& \sum_{k=1}^\infty \sum_{l=0}^\infty \sum_{i+g=k, i\geq1}\sum_{j+h=l, j\geq0} \frac{1}{i!g!}\frac{1}{j!h!}\gpke ij \perm{g+h-1} x^{i}x^g y^{j}y^h \\
 &=& \sum_{k=1}^\infty \sum_{i+g=k, i\geq1} \sum_{l=0}^\infty \sum_{j+h=l, j\geq0} \gpke ij \perm{g+h-1} \frac{x^{i}}{i!}\frac{x^g}{g!}\frac{y^{j}}{j!}\frac{y^h}{h!}\\
 &=& \sum_{i=1}^\infty\sum_{g=0}^\infty \sum_{j=0}^\infty\sum_{h=0}^\infty \gpke ij \frac{x^{i}}{i!}\frac{y^{j}}{j!} \perm{g+h-1} \frac{x^g}{g!}\frac{y^h}{h!}\\
 &=& \sum_{i=1}^\infty\sum_{g=0}^\infty \sum_{j=0}^\infty\sum_{h=0}^\infty \gpke ij \frac{x^{i}}{i!}\frac{y^{j}}{j!} \perm{g+h-1} \frac{x^g}{g!}\frac{y^h}{h!} \\
 &=& \sum_{i=1}^\infty \sum_{j=0}^\infty \sum_{g=0}^\infty \sum_{h=0}^\infty \gpke ij \frac{x^{i}}{i!}\frac{y^{j}}{j!} \perm{g+h-1} \frac{x^g}{g!}\frac{y^h}{h!}\\
&=& \left(\sum_{i=1}^\infty \sum_{j=0}^\infty  \gpke ij
\frac{x^{i}}{i!}\frac{y^{j}}{j!}\right)
 \left( \sum_{g=0}^\infty \sum_{h=0}^\infty \perm{g+h-1} \frac{x^g}{g!}\frac{y^h}{h!}\right) \\
&=& \gpkea (x,y) \permu(x+y).
\end{eqnarray*}

Recombining this we get
$$ d \gpkea (x,y) =  \gpkea (x,y)\left(y +  \permu(x+y)\right).$$
\end{proof}

Now that we have a formula for the derivative, we use it along with
equation \ref{did} to calculate the $f$-polynomial of this family.

\begin{lem}
We have that $\gpkeaf (x,y) =
e^{(\alpha+t)y}\frac{\eta(x)}{1-t\eta(x+y)}$.
\end{lem}
\begin{proof}

We can the pass to the face polynomial, we set
$\eta(z)=\frac{e^{\alpha z}-1}{\alpha}$ and we get
\begin{eqnarray*}
\frac{\partial}{\partial t}  \gpkeaf (x,y) &=&  \gpkeaf (x,y) \left(y+ \permu_f(x+y)\right) \\
&=&  \gpkeaf (x,y) \left(y- \frac{\partial}{\partial t}\ln(1-t\eta(x+y))\right) \\
\frac{\frac{\partial}{\partial t}  \gpkeaf (x,y)}{ \gpkeaf (x,y)}
&=&  \left(y- \frac{\partial}{\partial t}\ln(1-t\eta(x+y))\right) \\
\frac{\partial}{\partial t}  \ln\left(\gpkeaf (x,y)\right)
&=&  y- \frac{\partial}{\partial t}\ln(1-t\eta(x+y))\\
\gpkeaf (x,y) &=&  e^{yt}\frac{1}{1-t\eta(x+y)}c.
\end{eqnarray*}

Looking at $t=0$ we have initial conditions $\gpkeaf (\alpha,0,x,y)
=  e^{\alpha y}\eta(x)$ and so $$\gpkeaf (x,y) =
e^{(\alpha+t)y}\frac{\eta(x)}{1-t\eta(x+y)}.$$

Letting $y=0$ we have that $\gpkeaf (x,0) =
\frac{\eta(x)}{1-t\eta(x)} =\permu(x)$ as expected, since \gpke s0
=\perm s.
\end{proof}

Now that we have the $f$-polynomial of this family, we can repeat
theorem \ref{gpkeoisGal} for this family to show that it too
consists of Gal polytopes.

\begin{thm}
$\gpkeah (\alpha,t,x,y)$ is a Gal series.
\end{thm}
\begin{proof}
First we calculate the $h$ polynomial and we have
\begin{eqnarray*}
\gpkeah (x,y) &=& \gpkeaf ((\alpha-t),t,x,y) \\
&=& e^{\alpha y}e^{t y}\frac{e^{\alpha x}-e^{tx}}{\alpha e^{tx}e^{t
y}-te^{\alpha x}e^{\alpha y}} .
\end{eqnarray*}

So, we calculate the partial derivative or the $h$-polynomial with
respect to $x$,
\begin{eqnarray*}
\gpkeah (\alpha,t,x,y)&=& e^{(\alpha+t) y}\frac{e^{\alpha x}-e^{tx}}{\alpha e^{t(x+y)}-te^{\alpha (x+y)}} \\
\frac{\partial}{\partial x}\gpkeah (\alpha,t,x)&=& e^{(\alpha+t) y}\frac{\alpha e^{\alpha x}-t e^{tx}}{\alpha e^{t(x+y)}-te^{\alpha (x+y)}}-e^{(\alpha+t) y}\frac{(e^{\alpha x}-e^{tx})(\alpha te^{t(x+y)}-\alpha te^{\alpha (x+y)})}{(\alpha e^{t(x+y)}-te^{\alpha (x+y)})^2} \\
&=& e^{(\alpha+t) y}\phi_h (\alpha,t,x) + \alpha t \gpkeah
(\alpha,t,x,y) \permu_h (\alpha,t,x+y)
\end{eqnarray*}
where $\phi_h(\alpha,t,x,y)=\frac{\alpha e^{\alpha
x}-te^{tx}}{\alpha e^{t(x+y)}-te^{\alpha (x+ y)}}$. By \ref{keylem},
$\gpkeah (\alpha,t,x,y)$ is Gal if $\phi_h(\alpha,t,x,y)$ is Gal. We
now apply lemma \ref{keylem} to $\phi_h(\alpha,t,x,y)$
differentiating with respect to $y$ rather than $x$.
\begin{eqnarray*}
\frac{\partial}{\partial y}\frac{\alpha e^{\alpha x}-te^{tx}}{\alpha e^{t(x+y)}-te^{\alpha (x+ y)}}&=& \frac{-(\alpha e^{\alpha x}-te^{tx})(\alpha t e^{t(x+y)}-\alpha te^{\alpha (x+ y)})}{(\alpha e^{t(x+y)}-te^{\alpha (x+ y)})^2}\\
&=&\alpha t \permu_h(\alpha,t,x+y)\phi_h(\alpha,t,x,y)
\end{eqnarray*}
and we have that
\begin{eqnarray*}
\frac{\partial \phi_h(\alpha,t,x,y)}{\partial y}|_{y=0}&=&\frac{\alpha e^{\alpha x}-te^{tx}}{\alpha e^{tx}-te^{\alpha x}} \\
&=&1+(\alpha+t)\permu_h(\alpha,t,x)
\end{eqnarray*}
which is Gal. So by repeated application of the lemma, $\phi_\gamma$
is Gal and so is $\gpkeah(\alpha,t,x,y)$.
\end{proof}

\section{The Final Family, \ggeea}\label{sgpeea}

With these preliminaries over, we can finally move on to demonstrate
that the Gal conjecture holds for the family $\gpeea$, which is
generated by a family of graphs which in general is non-chordal. We
will follow the same steps as used in section \ref{sgpkea}. So we
start by finding a formula for the derivative of \gpeea.

\begin{lem}
We shall show that $d\gpeea(x,y)$ is
$$x\gpkea (y,x)  + y\gpkea (x,y)+ \gpeea(x,y)\permu(x+y)-(x+y)\permu(x+y).$$
\end{lem}

\begin{proof}
By expanding on the work of N. Erokhovets in \cite{NE} we have that,
for $s,t\geq 2$,
\begin{eqnarray*}
d\gpee st &=& s \gpek {s-1}t + t \gpke s{t-1}
+ \sum^{s-1}_{i=1}\sum^{t-1}_{j=1} {s \choose i}{t \choose j} \gpee ij \times \perm{s+t-i-j-1}+ \\
&& \sum^{s-1}_{i=1} {s \choose i} \gpee it \times \perm{s-i-1}+
\sum^{t-1}_{j=1} {t \choose j} \gpee sj \times \perm{t-j-1}.
\end{eqnarray*}
when either $s<2$ of $t<2$ there are only two possible outcomes.
Since, if $s=0$ then we must have $t=1$ for the graph to be
connected and vice versa, we must have either $s=1$ or $t=1$ or
both.  Here we notice that $\gpee 1k = \gpke 1k = \astr k$.

Let us now examine the generating function $\gpeea(x,y)$, since we
have two distinct indices the generating function is of two
variables. We have
\begin{eqnarray*}
\gpeea(x,y) &=& \sum^{\infty}_{n=0} \sum_{k+l=n+1,k,l \geq 0} \gpee kl \frac{x^k}{k!}\frac{y^l}{l!} \\
&=& \sum^{\infty}_{n=0} \sum_{k+l=n+1,k,l \geq 2} \gpee kl \frac{x^k}{k!}\frac{y^l}{l!}+\sum^{\infty}_{n=0} \sum_{k=n,l=1} \gpee k1 \frac{x^k}{k!}y\\
&& +\sum^{\infty}_{n=0} \sum_{l=n,k=1} \gpee 1l x\frac{y^l}{l!} - \gpee 11 xy \\
&=& \sum^{\infty}_{k=2}\sum^{\infty}_{l=2} \gpee kl \frac{x^k}{k!}\frac{y^l}{l!}+\sum^{\infty}_{k=0} \gpee k1 \frac{x^k}{k!}y +\sum^{\infty}_{l=0} \gpee 1l x\frac{y^l}{l!} - \gpee 11 xy \\
&=& \sum^{\infty}_{k=2}\sum^{\infty}_{l=2} \gpee kl
\frac{x^k}{k!}\frac{y^l}{l!}+ \gpkeo 1(x,y) + \gpkeo 1(y,x) - I^1
xy.
\end{eqnarray*}

so
\begin{eqnarray*}
d\gpeea(x,y) &=& d\sum^{\infty}_{k=2}\sum^{\infty}_{l=2} \gpee kl \frac{x^k}{k!}\frac{y^l}{l!}+ d\gpkeo 1(x,y) +  d\gpkeo 1(y,x) - d(I^1) xy \\
&=& \sum^{\infty}_{k=2}\sum^{\infty}_{l=2} \left( k \gpek {k-1} l +
l \gpke k {l-1}
+ \sum^{k-1}_{i=1}\sum^{l-1}_{j=1} {k \choose i}{l \choose j} \gpee ij \times \perm{k+l-i-j-1}\right. \\
&& +\left. \sum^{k-1}_{i=1} {k \choose i} \gpee il \times
\perm{k-i-1}+
\sum^{l-1}_{j=1} {l \choose j} \gpee kj \times \perm{l-j-1} \right) \frac{x^k}{k!}\frac{y^l}{l!}\\
&&+ d\gpkeo 1(x,y) + d\gpkeo 1(y,x) - d(I^1)xy \\
&=& \sum^{\infty}_{k=2}\sum^{\infty}_{l=2} \frac{x^k}{k!}\frac{y^l}{l!} k \gpek {k-1}l + \sum^{\infty}_{k=2}\sum^{\infty}_{l=2} \frac{x^k}{k!}\frac{y^l}{l!} l \gpke k{l-1}\\
&&+ \sum^{\infty}_{k=2}\sum^{\infty}_{l=2} \frac{x^k}{k!}\frac{y^l}{l!} \sum^{k-1}_{i=1}\sum^{l-1}_{j=1} {k \choose i}{l \choose j} \gpee ij \times \perm{k+l-i-j-1} \\
&& + \sum^{\infty}_{k=2}\sum^{\infty}_{l=2}
\frac{x^k}{k!}\frac{y^l}{l!} \sum^{k-1}_{i=1} {k \choose i} \gpee il
\times \perm{k-i-1}
+ \sum^{\infty}_{k=2}\sum^{\infty}_{l=2} \frac{x^k}{k!}\frac{y^l}{l!} \sum^{l-1}_{j=1} {l \choose j} \gpee kj \times \perm{l-j-1}  \\
&&+ (x+\permu(x))\gpkeo 1(x,y) + (y+\permu(y))\gpkeo 1(y,x) -
d(I^1) xy.
\end{eqnarray*}

Let us now consider each sum in turn.  Taking the first sum, we have
that
\begin{eqnarray*}
\sum^{\infty}_{k=2}\sum^{\infty}_{l=2} \frac{x^k}{k!}\frac{y^l}{l!} k \gpek {k-1}l &=& x\sum^{\infty}_{k=2}\sum^{\infty}_{l=2} \frac{x^{k-1}}{(k-1)!}\frac{y^l}{l!} \gpek {k-1}l \\
&=& x\sum^{\infty}_{k=1}\sum^{\infty}_{l=2} \frac{x^k}{k!}\frac{y^l}{l!} \gpek kl \\
&=& x\sum^{\infty}_{k=0}\sum^{\infty}_{l=1} \frac{x^k}{k!}\frac{y^l}{l!} \gpek kl \\
&&- x\sum^{\infty}_{k=1} \frac{x^k}{k!}y \gpek k1 -x\sum^{\infty}_{l=1} \frac{y^l}{l!} \gpek 0l\\
&=& x\gpkea (y,x) - x (\gpkeo 1(x,y)-y) - x\sum^{\infty}_{l=1}\perm{l-1} \frac{y^l}{l!} \\
&=& x\gpkea (y,x) - x (\gpkeo 1(x,y)-y) - x\permu(y).
\end{eqnarray*}
It is clear that the second sum is the same as the first sum with
$x$ and $y$ reversed.  Proceeding to the third sum, setting $g=k-i$
and $h=l-j$, we have
\begin{eqnarray*}
&& \sum^{\infty}_{k=2}\sum^{\infty}_{l=2} \frac{x^k}{k!}\frac{y^l}{l!} \sum^{k-1}_{i=1}\sum^{l-1}_{j=1} {k \choose i}{l \choose j} \gpee ij \times \perm{k+l-i-j-1}\\
&=& \sum^{\infty}_{k=2}\sum^{k-1}_{i=1}\sum^{\infty}_{l=2}  \sum^{l-1}_{j=1}\gpee ij \times \perm{k+l-i-j-1}\frac{x^i}{i!}\frac{x^{k-i}}{(k-i)!}\frac{y^j}{j!}\frac{y^{l-j}}{(l-j)!} \\
&=&\sum^{\infty}_{g=1}\sum^{\infty}_{i=1}\sum^{\infty}_{h=1}\sum^{\infty}_{j=1}\gpee ij \times \perm{g+h-1}\frac{x^i}{i!}\frac{x^g}{g!}\frac{y^j}{j!}\frac{y^h}{h!} \\
&=& \left(\sum^{\infty}_{i=1}\sum^{\infty}_{j=1}\gpee ij\frac{x^i}{i!}\frac{y^j}{j!}\right) \left(\sum^{\infty}_{g=1}\sum^{\infty}_{h=1}  \perm{g+h-1}\frac{x^g}{g!}\frac{y^h}{h!}\right) \\
&=& \left(\gpeea(x,y)-(x+y)\right) \left(\permu(x+y) -\sum^{\infty}_{g=1}  \perm{g-1}\frac{x^g}{g!}-\sum^{\infty}_{h=1}  \perm{h-1}\frac{y^h}{h!}\right)\\
&=& \left(\gpeea(x,y)-(x+y)\right) \left(\permu(x+y)
-\permu(x)-\permu(y)\right).
\end{eqnarray*}

Again we notice the similarities between the fourth and fifth sums.
We examine the fourth in detail, again with $g=k-i$.
\begin{eqnarray*}
\sum^{\infty}_{k=2}\sum^{\infty}_{l=2} \frac{x^k}{k!}\frac{y^l}{l!}
\sum^{k-1}_{i=1} {k \choose i} \gpee il \times \perm{k-i-1}&=&
\sum^{\infty}_{k=2}\sum^{k-1}_{i=1}\sum^{\infty}_{l=2}  \gpee il \times \perm{k-i-1} \frac{x^{k-i}}{(k-i)!}\frac{x^i}{i!}\frac{y^l}{l!}  \\
&=& \sum^{\infty}_{g=1}\sum^{\infty}_{i=1}\sum^{\infty}_{l=2}  \gpee il \times \perm{g-1} \frac{x^g}{g!}\frac{x^i}{i!}\frac{y^l}{l!} \\
&=& \left(\sum^{\infty}_{i=1}\sum^{\infty}_{l=2}  \gpee il \frac{x^i}{i!}\frac{y^l}{l!} \right)\left(\sum^{\infty}_{g=1}\perm{g-1} \frac{x^g}{g!}\right)\\
&=& \left(\sum^{\infty}_{n=0} \sum_{i+l=n+1,i,l \geq 0}  \gpee il
\frac{x^i}{i!}\frac{y^l}{l!} -\sum^{1}_{i=1}\sum^{0}_{l=0}  \gpee il
\frac{x^i}{i!}\frac{y^l}{l!}\right. \\&&
\left.-\sum^{\infty}_{i=0}\sum^{1}_{l=1}  \gpee il
\frac{x^i}{i!}\frac{y^l}{l!} \right)
\left(\sum^{\infty}_{g=1}\perm{g-1} \frac{x^g}{g!}\right)\\
&=& \left(\gpeea (x,y) - x - \gpkeo 1(x,y)
\right)\left(\permu(x)\right).
\end{eqnarray*}

We can combine all of these to get
\begin{eqnarray*}
d\gpeea(x,y) &=& \sum^{\infty}_{k=2}\sum^{\infty}_{l=2} \frac{x^k}{k!}\frac{y^l}{l!} k \gpek {k-1}l + \sum^{\infty}_{k=2}\sum^{\infty}_{l=2} \frac{x^k}{k!}\frac{y^l}{l!} l \gpke k{l-1}\\
&&+ \sum^{\infty}_{k=2}\sum^{\infty}_{l=2} \frac{x^k}{k!}\frac{y^l}{l!} \sum^{k-1}_{i=1}\sum^{l-1}_{j=1} {k \choose i}{l \choose j} \gpee ij \times \perm{k+l-i-j-1} \\
&& + \sum^{\infty}_{k=2}\sum^{\infty}_{l=2} \frac{x^k}{k!}\frac{y^l}{l!} \sum^{k-1}_{i=1} {k \choose i} \gpee il \times \perm{k-i-1}\\
&&+ \sum^{\infty}_{k=2}\sum^{\infty}_{l=2} \frac{x^k}{k!}\frac{y^l}{l!} \sum^{l-1}_{j=1} {l \choose j} \gpee kj \times \perm{l-j-1}  \\
&&+  (x+\permu(x))\gpkeo 1(x,y) + (y+\permu(y))\gpkeo 1(y,x) - d(I^1) xy \\
&=& x\gpkea (y,x) - x(\gpkeo 1(x,y)-y) - x\permu(y) + y\gpkea (x,y) - y (\gpkeo 1(y,x)-x) - y\permu(x)\\
&&+  \left(\gpeea(x,y)-(x+y)\right) \left(\permu(x+y) -\permu(x)-\permu(y)\right)\\
&& + \left(\gpeea (x,y) - x - \gpkeo 1(x,y) \right)\left(\permu(x)\right)+ \left(\gpeea (x,y) - y - \gpkeo 1(y,x) \right)\left(\permu(y)\right)  \\
&&+ (x+\permu(x))\gpkeo 1(x,y) + (y+\permu(y))\gpkeo 1(y,x) -  d(I^1) xy \\
&=& x\gpkea (y,x)  + y\gpkea (x,y) +
\gpeea(x,y)\permu(x+y)-(x+y)\permu(x+y).
\end{eqnarray*}
\end{proof}

Now that we have the formula expressing the boundary of this series
of polytopes we can, as before, use it to calculate the
$f$-polynomial of the series. Solving the appropriate differential
equation, as set out below, gives us:

\begin{lem}
We have that, with $\eta(x)$ as before,
 $\gpeeaf(\alpha,t,x,y)$ is $$\frac{1}{1-t\eta(x+y)}\left( e^{(\alpha+t)x}\eta(y)+e^{(\alpha+t)y}\eta(x)+\alpha \eta(y)\eta(x)-e^{\alpha x}\eta(y)-e^{\alpha y}\eta(x)\right)+(x+y).$$
\end{lem}
\begin{proof}
By the identity \ref{did} we have
\begin{eqnarray*}
\frac{\partial}{\partial t}\gpeeaf(x,y) &=& x\gpkeaf (y,x)  + y\gpkeaf (x,y)\\
&&+  \gpeeaf(x,y)\permu_f(x+y)-(x+y)\permu_f(x+y) \\
&=& xe^{(\alpha+t)x}\frac{\eta(y)}{1-t\eta(x+y)}  + ye^{(\alpha+t)y}\frac{\eta(x)}{1-t\eta(x+y)}\\
&&+
\gpeeaf(x,y)\frac{\eta(x+y)}{1-t\eta(x+y)}-(x+y)\frac{\eta(x+y)}{1-t\eta(x+y)}.
\end{eqnarray*}

To solve this we shall start by setting $\hat P = \gpeeaf -(x+y)$,
then we have
$$
\frac{\partial}{\partial t}\hat P(x,y) =
\frac{xe^{(\alpha+t)x}\eta(y)+ye^{(\alpha+t)y}\eta(x)+\hat
P(x,y)\eta(x+y)}{1-t\eta(x+y)}.
$$

If we now set $\hat P=P_1P_2$ and $P_1=\frac{c_1}{1-t\eta(x+y)}$
then we get, by application of the quotient rule and integrating,
$$P_2=e^{(\alpha+t)x}\eta(y)+e^{(\alpha+t)y}\eta(x)+c_2.$$

combining all these we have
\begin{eqnarray*}
P_2(\alpha,t,x,y) &=& e^{(\alpha+t)x}\eta(y)+e^{(\alpha+t)y}\eta(x)+c_2 \\
\hat P(\alpha,t,x,y) &=& \frac{c_1}{1-t\eta(x+y)}\left( e^{(\alpha+t)x}\eta(y)+e^{(\alpha+t)y}\eta(x)+c_2\right) \\
\gpeeaf(\alpha,t,x,y) &=& \frac{c_1}{1-t\eta(x+y)}\left( e^{(\alpha+t)x}\eta(y)+e^{(\alpha+t)y}\eta(x))+c_2\right)\\
&&+(x+y).
\end{eqnarray*}

Examining the initial conditions we have that
\begin{eqnarray*}
\gpeeaf (\alpha,0,x,y) &=& \sum_{n=0}^{\infty}\sum_{k+l=n+1,k,l>0}\alpha^n \frac{x^k}{k!}\frac{y^l}{l!}+x+y\\
&=& \sum_{n=0}^{\infty}\sum_{k+l=n+1,k\geq1,l\geq 0}\alpha^n \frac{x^k}{k!}\frac{y^l}{l!}-\sum_{n=0}^{\infty}\sum_{k+l=n+1,k>0,l=0}\alpha^n \frac{x^k}{k!}+x+y\\
&=& \sum_{n=0}^{\infty}\sum_{k+l=n+1,k\geq1,l\geq 0}\alpha^n \frac{x^k}{k!}\frac{y^l}{l!}-\sum_{k=1}^{\infty}\alpha^{k+1} \frac{x^k}{k!}+x+y\\
&=& \gpkeaf (\alpha,0,x,y)-\eta(x)+x+y\\
&=& e^{\alpha y}\eta(x)-\eta(x)+x+y\\
&=& \alpha \eta(y)\eta(x)+x+y\
\end{eqnarray*}
so, setting $c_1=1$, we have
\begin{eqnarray*}
\alpha \eta(y)\eta(x)+x+y &=& \frac{1}{1}\left( e^{\alpha x}\eta(y)+e^{\alpha y}\eta(x)+c_2\right)+(x+y)\\
c_2 &=& \alpha \eta(y)\eta(x)-e^{\alpha x}\eta(y)-e^{\alpha
y}\eta(x).
\end{eqnarray*}

Then we have
\begin{eqnarray*}
\gpeeaf(\alpha,t,x,y) &=& \frac{1}{1-t\eta(x+y)}\left( e^{(\alpha+t)x}\eta(y)+e^{(\alpha+t)y}\eta(x)\right.\\
&&\left.+\alpha \eta(y)\eta(x)-e^{\alpha x}\eta(y)-e^{\alpha
y}\eta(x)\right)+(x+y).
\end{eqnarray*}
\end{proof}

With the $f$-polynomial of the family now calculated we can
demonstrate that \gpeea is a family of polytopes that satisfy the
Gal conjecture. We will do this in the same way we did for the
families \gpkeo 1 and \gpkea, using lemma \ref{keylem}. Unlike in
the previous sections we can no longer use the result from
\cite{PRW} since for $n,m\geq 2$, $\ggee nm$ is not a chordal graph.

\begin{thm}
$\gpeeah (\alpha,t,x,y)$ is a Gal series.
\end{thm}
\begin{proof}
As before we start by calculating the series
$\gpeeah(\alpha,t,x,y)$.
\begin{eqnarray*}
\gpeeah(\alpha,t,x,y) &=& \gpeeaf((\alpha-t),t,x,y) \\
&=& \frac{1}{1-t\frac{e^{(\alpha-t) (x+y)}-1}{(\alpha-t)}}\left( e^{((\alpha-t)+t)x}\frac{e^{(\alpha-t) (y)}-1}{(\alpha-t)}+e^{((\alpha-t)+t)y}\frac{e^{(\alpha-t) (x)}-1}{(\alpha-t)}\right.\\
&&\left.+(\alpha-t) \frac{e^{(\alpha-t)(y) }-1}{(\alpha-t)}\frac{e^{(\alpha-t) (x)}-1}{(\alpha-t)}-e^{(\alpha-t) x}\frac{e^{(\alpha-t) (y)}-1}{(\alpha-t)}\right.\\
&&\left.-e^{(\alpha-t) y}\frac{e^{(\alpha-t) (x)}-1}{(\alpha-t)}\right)+(x+y)\\
&=& \frac{1}{\frac{\alpha e^{tx}e^{ty}-te^{\alpha x}e^{\alpha
y}}{(\alpha-t)}} \left( \frac{e^{\alpha y}e^{\alpha
x}e^{tx}-e^{\alpha x}e^{tx}e^{ty}
+e^{\alpha x}e^{\alpha y}e^{ty}-e^{\alpha y}e^{tx}e^{ty}}{(\alpha-t)}\right.\\
&&\left.+\frac{e^{tx}e^{ty}-e^{\alpha x}e^{\alpha
y}}{(\alpha-t)}\right)+(x+y).
\end{eqnarray*}

So, looking at the partial derivative of $\gpeeah(\alpha,t,x,y)$
with respect to $x$, we have,
\begin{eqnarray*}
\frac{\partial}{\partial x}\gpeeah(\alpha,t,x,y) &=&
\frac{(\alpha-t)}{\alpha t e^{tx}e^{ty}-\alpha te^{\alpha
x}e^{\alpha y}}\\&& \left( \frac{e^{\alpha y}e^{\alpha
x}e^{tx}-e^{\alpha x}e^{tx}e^{ty}
+e^{\alpha x}e^{\alpha y}e^{ty}-e^{\alpha y}e^{tx}e^{ty}}{(\alpha-t)}\right.\\
&&\left.+\frac{e^{tx}e^{ty}-e^{\alpha x}e^{\alpha y}}{(\alpha-t)}+\frac{\alpha e^{tx}e^{ty}-te^{\alpha x}e^{\alpha y}}{(\alpha-t)}(x+y)\right) \\
&&+\frac{(\alpha-t)}{\alpha e^{tx}e^{ty}-te^{\alpha x}e^{\alpha y}}
\frac{1}{(\alpha-t)}\left( (\alpha+t)e^{\alpha y}e^{\alpha
x}e^{tx}-(\alpha+t)e^{\alpha x}e^{tx}e^{ty}\right.\\&&\left.
+\alpha e^{\alpha x}e^{\alpha y}e^{ty}-t e^{\alpha y}e^{tx}e^{ty}\right.\\
&&\left.+t e^{tx}e^{ty}-\alpha e^{\alpha x}e^{\alpha y}+(\alpha  e^{tx}e^{ty}- te^{\alpha x}e^{\alpha y})+(\alpha t e^{tx}e^{ty}-\alpha te^{\alpha x}e^{\alpha y})(x+y)\right)\\
&=&\alpha t \permu_h(\alpha,t,x+y) \gpeeah + (\alpha+t)\gpkeah(\alpha,t,y,x)\\
&& +(\alpha +t+\alpha t
(x+y))\permu_h(\alpha,t,x+y)+e^{(\alpha+t)y}\frac{(\alpha e^{\alpha
x}-te^{tx})}{\alpha e^{t(x+y)}-te^{\alpha (x+ y)}}.
\end{eqnarray*}

So by our lemma, \gpeeah is a Gal series if
$\phi_h(\alpha,t,x,y)=\frac{\alpha e^{\alpha x}-te^{tx}}{\alpha
e^{t(x+y)}-te^{\alpha (x+ y)}}$ is a Gal series.  We showed that
this series was Gal in the previous section, so \gpeeag is a Gal
series by \ref{keylem}.
\end{proof}

\end{document}